\documentclass[twocolumn]{autart}    

\usepackage{graphicx}          
                               
\usepackage{bm}
\usepackage{amssymb}
\usepackage{apacite}
\usepackage{natbib}
\usepackage{enumerate}
\usepackage{mathrsfs}
\usepackage{amsmath}
\usepackage{mathtools}
\usepackage[table]{xcolor}
\usepackage[colorlinks,citecolor=blue!80,urlcolor=blue,bookmarks=false,hypertexnames=true]{hyperref} 

\newcommand{\hd}{d_{\Delta}}

\newcommand{\ones}{\mathbf{1}}

\newcommand{\reals}{\mathbb{R}}

\newcommand{\integers}{\mathbb{Z}}

\newcommand{\Jf}{\mathfrak{I}}

\newcommand{\T}{\top}

\newcommand*\xbar[1]{%
  \hbox{%
     \vbox{%
      \hrule height 0.7pt 
      \kern0.35ex
      \hbox{%
         \kern-0.1em
         \ensuremath{#1}%
         \kern-0.1em
      }%
     }%
  }%
} 

\newenvironment{proof}{\textbf{Proof.~}}{\hfill$\square$}

\begin{document}

\begin{frontmatter}

\title{Dissipativity-based static output feedback design for discrete-time LTI systems with time-varying input delays\thanksref{footnoteinfo}}
                                           \author{Thiago Alves Lima*}\ead{thiago.lima@alu.ufc.br},    
\author{Diego de S. Madeira}\ead{dmadeira@dee.ufc.br}  

\thanks[footnoteinfo]{This paper was not presented at any conference. Thanks to the Brazilian agency CAPES for financial support.\\ $^\ast$ Corresponding author.}

\address{Departamento de Engenharia Elétrica, Universidade Federal do Ceará, Fortaleza, Brazil.}  

\begin{keyword}                           
Static output feedback; Time-varying delays; Stabilization; Discrete-time systems; Dissipativity-based control.               
\end{keyword}                             

\begin{abstract}                          
This note is concerned with the presentation of new delay-dependent dissipativity-based convex conditions (expressed in the form of linear matrix inequalities) for the design of static output feedback (SOF) stabilizing gains for open-loop unstable discrete-time systems with input time-varying delays. A modified definition of QSR-dissipativity combined with the use of Lyapunov-Krasovskii functionals as storage functions and the application of Finsler's Lemma lead to the gathering of non-interactive design conditions. We show that, differently from most works dealing with controller design for time-delayed systems, the developed conditions present very small conservatism compared to stability analysis conditions derived with the same strategy. Due to being a particular case of SOF with an identity output matrix, static state feedback (SSF) gains can also trivially be computed from the conditions.
\end{abstract}

\end{frontmatter}

\section{Introduction}

The static output feedback stabilization problem has been widely investigated in the control literature \citep{SADABADI2016}. There exists a consensus among the control research community that a definitive answer is yet to be found even for the simple case of linear time-invariant (LTI) systems. Most of the proposed solutions are either interactive \citep{CAO1998}, impose conservative model transformations and/or rank constraints on the plant matrices \citep{Crusius_1999}. The reader is encouraged to read the survey papers \cite{SADABADI2016,SYRMOS1997} for a view of the several existing strategies for computing SOF gains. 

On the other hand, time delay is another problem difficulty to be dealt with since its presence can degrade closed-loop stability and performance properties \citep{Fridman_2014}. One strategy to study stability and stabilization of such systems is by means of the proposal and manipulation of so-called Lyapunov-Krasovskii functionals (LKFs). For such systems, even full-state feedback control design can be challenging since sufficient and necessary convex design conditions are yet to be found, contrary to the non-delayed LTI case. Some solutions to the state feedback design can be found in \cite{Suplin2006,SUPLIN2007,Wu2010,FRIDMAN2004,Shaked_2002,Zeng_2015,ZHANG2005,CHEN2006}. Nonetheless, all these solutions are based either in iterative approaches or on the linearization of bilinear matrix inequalities (BMIs) by means of conservative transformations and/or introduction of slack variables.  

The static output feedback stabilization of input time-delayed systems is, of course, even more evolved as it combines both the difficulties related to the already complicated SOF design of LTI systems and the ones introduced by the delay itself. The fact that the survey papers \cite{SADABADI2016,SYRMOS1997} have not even mentioned the SOF problem for delayed systems further enlightens how unexplored this problem still is. Nonetheless, one can cite the iterative approaches proposed in \cite{DU2010} and \cite{Barreau_2018} for continuous-time delayed systems. 

In this note, we propose a solution to the static output feedback stabilization of linear discrete-time systems with input time-varying delays that can arbitrarily vary between maximum and minimum bounds. As long as the authors know, this is the first work proposing a solution for the SOF design for this class of system. The proposal is inspired by the ideas recently published in \cite{Madeira_2021}, where necessary and sufficient conditions for the exponential stabilizability of nonlinear non-delayed systems by linear SOF was presented with the aid of dissipativity theory. In order to develop the conditions in this note, an appropriate modification to the definition of strict QSR-dissipativity used in \cite{Madeira_2021} is introduced to deal with the case of LTI systems with input time-varying delays, along with the use of a LKF (denoted $V(k)$ in the paper) as storage function and the application of a second inequality that ensures negativity of its forward difference $\Delta V(k)$ along the trajectories of the closed loop with the designed control gain. The design conditions are expressed in the form of linear matrix inequalities (LMIs), are non-interactive, and no restriction on the plant output matrix is assumed.


\textbf{Notation.} For a real matrix $Y $ in $ \reals^{n \times m}$, $Y^{\T} $ in $ \reals^{m \times n}$ means its transpose, $Y_{\left(i\right)}$ denotes its $i$th row, while for $v $ in $ \reals^{m}$, $v_{\left(i\right)}$ denotes its $i$th component. For matrices $W = W^{\T}$ and $Z = Z^{\T} $ in $ \reals^{n \times n}$, $W \succ Z$ means that $W-Z$ is positive definite. Likewise, $W \succeq Z$ means that $W-Z$ is positive semi-definite. $\mathbb{S}_n^{+}$ stands for the set of symmetric positive definite matrices. $\ones$, $I$, and $0$ denote all-ones, identity, and null matrices of appropriate dimensions, although their dimensions can be explicitly presented whenever relevant. The $\star$ in the expression of a matrix denotes symmetric blocks. For matrices $W$ and $Z$, \textit{diag}$(W,Z)$ corresponds to the block-diagonal matrix. Finally, $\reals ^+$ denotes the set of elements $\beta \in \reals$ such that $\beta \geq 0$.

\section{Problem formulation}
\label{sec:pb}

Consider the plant described by the following equations
\begin{align}
    \begin{cases}\label{eq:Plant}
      x(k+1)=A x(k)+B  u(k-d(k)) \\
      y(k)= C x(k)\\
      u(k) = \theta_u(k),~k\in[-d_M,-1]
    \end{cases} 
    \end{align}
\noindent where $x(k) \in \reals ^{n}$ is the plant state vector, $y(k) \in \reals^{p}$ is the delayed measured output, $u(k) \in \reals^{m}$ is the control input, and $x(0)$, $\theta_u(k)$ are initial conditions for the state and control input. The pair ($A$, $B$) with constant and known $A$, $B$ matrices of appropriate dimensions is controllable. The plant output delay is bounded and time-varying such as $1 \leq d_m \leq d(k) \leq d_M$, and can arbitrarily vary within such limits. Furthermore, integers $d_m$ and $d_M$ are known, whereas the value of $d(k)$ at each sampling time is unknown. 

To control \eqref{eq:Plant} we consider the following static output feedback control law
\begin{align}\label{eq:control}
      u(k) = K y(k),
    \end{align}
\noindent where $K \in \reals^{m \times p}$ is a stabilizing gain to be designed. The interconnection \eqref{eq:Plant}-\eqref{eq:control} generates the following state-delayed closed-loop system
\begin{align}
   \begin{cases}
              x(k+1)=A x(k) + B K C x(k-d(k))\\
          x(k) = \phi(k), ~k \in \left[-d_M,0 \right] 
          \end{cases}\label{clsystem}
\end{align}
where $\phi(k)$ is the initial condition at the interval $\left[-d_M,0 \right]$. Then the problem we intend to solve in this note can be summarized as follows.
\begin{prob}
Given the plant matrices $A$, $B$, $C$, and the time-varying delay limits $d_m$, $d_M$, develop convex conditions in the form of linear matrix inequalities (LMIs) for the design of a matrix $K$ such that the asymptotic stability of the closed-loop system given by the interconnection \eqref{eq:Plant}-\eqref{eq:control} (i.e., system \eqref{clsystem}) is ensured for any $d_m \leq d(k) \leq d_M$.
\label{prob1}
\end{prob}

\section{Theoretical preliminaries}
Some important theoretical preliminaries are shortly reviewed in this section.

\subsection{Stability of time-delayed systems}
In general, stability of time-delayed systems can be tackled by using either delay-independent or delay-dependent conditions. The latter case (in which bounds on the delay are explicitly considered) is preferred in this work. Let us consider discrete-time linear system \eqref{clsystem} with time-varying delay $d(k)$. Let us define the state of \eqref{clsystem} as $\xbar{x}(k) \triangleq x(k+\theta)$, $\theta \in \{-d_M,-d_M+1,\dots,0\}$. In the Lyapunov-Krasoviskii framework, system \eqref{clsystem} is asymptotically stable if, for all $k \in \integers^{+}= \{0,1,2,\dots\}$, there exists a functional 
\begin{equation}\label{eq:LKfunctional}
    V(k) \triangleq
V(k,\xbar{x}(k)):\integers^{+}\times \overbrace{\reals ^n \times \dotsc \times \reals ^n}^{d_M+1 \text{ times}} \rightarrow \reals ^+
\end{equation}
such that, for all $\xbar{x}(k) \neq 0$, $V(k) > 0$ and its forward difference is negative, i.e., $\Delta V(k) = V(k+1)-V(k) < 0$. 

\subsection{Dissipativity conditions}

A non-delayed nonlinear dynamical system such as
\begin{align}
    \begin{cases}\label{nonlinearsystem}
x(k+1)=f(x(k))+g(x(k))u(k),\\
y(k)=Cx(k),
    \end{cases}
\end{align}
with state $x(k) \in \reals^{n}$ is said to be dissipative if there exists a continuous nonnegative storage function $V(x(k)): \reals^n \rightarrow \reals^{+}$ and a so-called supply rate $w(k)=w(u(k),y(k))$ such that for all $u \in \reals^{m}$ and all $k \in \integers^{+}$ the relation 
\begin{equation*}
    V(x(k+1))-V(x(k)) \leq w(u(k),y(k)),
\end{equation*}
holds \citep{Byrnes_1994}. Some definitions of dissipativity can be found in \cite{brogl1}. In this work, we present the definition of strict QSR-dissipativity below.
\begin{defn}\label{defdissip}
A system is said to be strictly QSR-dissipative along all possible trajectories of \eqref{nonlinearsystem} starting at \(x(0)\), for all \(k \geq 0\), if there exists \(T(x(k))>0\) such that
\begin{equation}\label{dissipativityeq_geral}
 \Delta V(x(k))+T(x(k)) \leq w(u(k),y(k)),
\end{equation}
\noindent with supply rate
\begin{equation}
    w(u(k),y(k)) = y^{\top}Qy+2y^{\top}Su+u^{\top}Ru,
\end{equation}
where \(S \in \mathbb{R}^{p \times m}\), and matrices \(Q \in \mathbb{R}^{p \times p}\) and \(R \in \mathbb{R}^{m \times m}\) are symmetric.
\end{defn}

Although being not the same in essence, when doing dissipativity-based control, the storage function is often regarded as the Lyapunov function so that stability of equilibrium points can be simultaneously investigated along with the dissipativity properties. In this work, we are interested in studying stabilization of system \eqref{eq:Plant} with control law \eqref{eq:control}, which in turn generates state-delayed closed-loop system \eqref{clsystem}. In order to properly establish delay-dependent stability of \eqref{clsystem}, we will consider a modified version of Definition \ref{defdissip}, which falls within the case of non-strict dissipativity. 

\begin{defn}\label{def:dissipatividelayed}
Consider a storage function given by a functional \eqref{eq:LKfunctional} of an augmented state $\xbar{x}(k)$. If 
\begin{equation}\label{eq:dissipativitymain}
\Delta V(k) \leq w(u(k-d(k)),y(k-d(k)))
\end{equation}
\noindent holds with supply rate
\begin{equation}\label{eq:suplydelay}
    \begin{split}
        &w(u(k-d(k)),y(k-d(k))) = \\ &~~~~~~~~~~y(k-d(k))^{\top}Qy(k-d(k)) \\
     &~~~~~~~~~~~+2y^{\top}(k-d(k))Su(k-d(k))\\&~~~~~~~~~~~~~~~~~~~~~+u^{\top}(k-d(k))Ru(k-d(k))
    \end{split},
\end{equation} 
then system \eqref{eq:Plant} is dissipative with respect to the delayed input-output pair $u(k-d(k)),y(k-d(k))$. 
\end{defn}

Naturally, application of Definition \ref{def:dissipatividelayed} is useful to determine about the dissipativity properties of the open-loop plant \eqref{eq:Plant}. Moreover, if at the same time that \eqref{eq:dissipativitymain} is fulfilled we can ensure that $\Delta V(k)$ is negative definite with system \eqref{eq:Plant} subject to the control law \eqref{eq:control}, then the uniform asymptotic stability of \eqref{clsystem} can be guaranteed for any time-varying delay $d_m \leq d(k) \leq d_M$. Such requirement can be satisfied by means of a second condition developed using Finsler's Lemma. This is the key idea used in this work, as will become clearer in Section \ref{sec:main}. 


\subsection{Finsler's Lemma}
Prior to the presentation of the main result, the celebrated Finsler's Lemma is reviewed.

\begin{lem}\label{lemma:Finsler} \citep{Mauricio_2001}
Consider $\zeta \in \reals^{n_{\zeta}}$, ${\Phi}={\Phi}^{\T} \in \reals^{n_{\zeta} \times n_{\zeta}}$, and $\Gamma \in \reals^{m_{\zeta} \times n_{\zeta}}$. The following facts are equivalent:

\begin{enumerate}[(i)]
    \item  $\zeta^{\T} {\Phi} \zeta<0$, $\forall \zeta$ such that $\Gamma\zeta=0$, $\zeta \neq 0$.
    \item $\exists \Jf \in \reals^{n_{\zeta} \times m_{\zeta}}$ such that ${\Phi}+\Jf \Gamma+\Gamma^{\T}\Jf^{\T} \prec 0$.
    \item $\Gamma^{\perp^{\T}} {\Phi} \Gamma^{\perp} \prec 0$, where $\Gamma \Gamma^{\perp} = 0$. 
\end{enumerate}
\end{lem}

\section{Main results}\label{sec:main}
In this section, we present convex conditions for the design of the stabilizing gain $K$. 

\subsection{Controller design}
Let us define $\hd=d_M-d_m$, and the function $\gamma(d)$ given by
\begin{equation}\label{eq:gammafunc}
    \begin{cases}
            \gamma(d)=1, \text{ if } d=1 \\
            \gamma(d)=(d+1)/(d-1), \text{ if } d>1.
    \end{cases}
\end{equation}
\noindent Then, the following theorem provides a solution to Problem \ref{prob1}.
\begin{thm} \label{designtheorem}
Given a scalar $\rho$, assume that there exist matrices ${P} $ in $\mathbb{S}_{3n}^{+}$, ${W}_1$, ${W}_2$, ${Z}_1$, ${Z}_2$, $R$ in $\mathbb{S}_{m}^{+}$, symmetric matrix $Q$ in $\reals^{p \times p}$, matrices $S$ in $\reals^{p \times m}$ and ${X}$ in $\reals^{2n\times2n}$ such that 
\begin{equation} \label{main_matineq}
\begin{split}
    {\Psi}_z \succ 0, ~~\Gamma^{ \perp^{\T}}\left({\xbar{\Phi}}(d)+\Upsilon \right) \Gamma^{ \perp} \prec 0,
\end{split}
\end{equation}
\begin{equation}\label{main_stab}
     \begin{bmatrix}
    Q + \rho \ones_{p \times m} S^\T + S \rho \ones_{m \times p} & \star  \\
    R^\T \rho \ones_{m \times p} & -R
    \end{bmatrix} \prec 0,
\end{equation}
\noindent hold for $d=d_m$ and $d=d_M$, with $\Gamma^{\perp} = \begin{bmatrix} A & 0_{n \times 6n} & B \\
& I_{7n+m}& \end{bmatrix}$,
\begin{equation*}
\begin{split}
    &\Upsilon = \Upsilon_1 + \Upsilon_2 + \Upsilon_2^{\T}, \\
    &{\Upsilon_1} = \text{diag}(0_{3n \times 3n},-C^{\T} Q C,0_{4n \times 4n},-R), \\
    & \Upsilon_2 = \begin{bmatrix}
        0_{n \times 3n} & I & 0_{n \times (4n+m)}
    \end{bmatrix}^{\T} \begin{bmatrix}
        0_{n \times 8n} & -C^{\T} S
    \end{bmatrix},
\end{split}
\end{equation*}
\noindent and
\begin{equation*}
\begin{split}
  \xbar{\Phi}(d)&=diag\{\Phi(d),0_{m \times m}\}, \\
    {\Phi}(d) &= F_{2}^{\T}{P}F_{2}  -F_{1}^{\T}{P}F_{1} \hspace{-0.05cm} + \hspace{-0.05cm} \text{He}\{F^{\T}(d)P(F_2-F_1)\}\hspace{-0.05cm}+\hspace{-0.05cm}{W} \\
   +&F_3^{\T}(d_m^2 {{Z}}_1+\hd^2 {{Z}}_2) F_3-F_{s}^{\T} \mathscr{Z}_1(d_m)F_{s}-F_{\Psi}^{\T}{\Psi}_zF_{\Psi},
\end{split}
\end{equation*}
\noindent where 
\begin{equation*}
    {W} = \text{diag}(0,{W}_1,{W}_2-{W}_1,0,-{W}_2,0,0,0),
\end{equation*}
\begin{equation*}
    \mathscr{Z}_1(d_m) = \text{diag}\left({Z}_1,3\gamma(d_m){Z}_1\right), \mathscr{Z}_2 = \text{diag}\left({Z}_2,3{Z}_2\right),
\end{equation*}
\begin{equation*}
    F_{\Psi} = \begin{bmatrix}
   0_{2n\times2n} & & M \\ 0_{2n\times n} & M & 0_{2n\times n}
   \end{bmatrix}, M = \begin{bmatrix}  0 & I & -I & 0 & 0 & 0 \\
   0 & I & I & 0 & 0 & -2I \end{bmatrix},
\end{equation*}
\begin{equation*}
    F_{s} = \begin{bmatrix} M & 0_{2n\times2n} \end{bmatrix}, F_{3} = \begin{bmatrix} I & -I & 0 & 0 & 0 & 0 & 0 & 0 \end{bmatrix},
\end{equation*}
\begin{equation*}
    F_{1} = \begin{bmatrix} 0 & I & 0 & 0 & 0 & 0 & 0 & 0 \\
     0 & -I & 0 & 0 & 0 & (d_m +1)I & 0 & 0 \\
     0 & 0 & -I & -I & 0 & 0& (1-d_m)I & (d_M+1)I \end{bmatrix},
\end{equation*}
\begin{equation*}
    F_{2} = \begin{bmatrix} I & 0 & 0 & 0 & 0 & 0 & 0 & 0 \\
     0 & 0 & -I & 0 & 0 & (d_m +1)I & 0 & 0 \\
     0 & 0 & 0 & -I & -I & 0& (1-d_m)I & (d_M+1)I \end{bmatrix},
\end{equation*}
\begin{equation*}
F(d) = \begin{bmatrix}
0_{2n \times 6n} & 0_{2n \times n} & 0_{2n \times n} \\
0_{n \times 6n} & d I_n & -d I_n
\end{bmatrix}, {\Psi}_z = \begin{bmatrix}
    ~\mathscr{Z}_2 & {X}~ \\
    ~\star & \mathscr{Z}_2~
    \end{bmatrix}.
\end{equation*}
\noindent Then, the SOF gain $K=-R^{-1}S^{\T}$ asymptotically stabilizes the closed-loop system \eqref{clsystem} for any time-varying delay $d_m \leq d(k) \leq d_M$. 
\end{thm}

\begin{proof}
In order to establish stability of \eqref{clsystem} for any time-varying delay $d_m \leq d(k) \leq d_M$,  we need to find a Lyapunov-Krasoviskii functional $V(k)$ that satisfies $V(k+1)-V(k)<0$. Consider then the LKF borrowed from \cite{Seuret_2015}, given by
\begin{equation}\label{LKFFunc}
    V(k) = V_{1}(k)+V_{2}(k)+V_{3}(k),
\end{equation}
\noindent where
\begin{equation*}
\begin{split}
&V_{1}(k) = w^{\T}(k)P w(k), \\
&V_{2}(k) = \sum_{l=k-d_m}^{k-1} x^{\T}(l)W_1 x(l)+\sum_{l=k-d_M}^{k-d_m-1} x^{\T}(l)W_2x(l), \\
&\!\begin{aligned}
V_{3}(k) = & d_m  \sum_{l=-d_m+1}^{0} \sum_{i=k+l}^{k} \eta^{\T}(i)Z_1\eta(i) \\ &+  \hd \sum_{l=-d_M+1}^{-d_m} \sum_{i=k+l}^{k} \eta^{\T}(i)Z_2\eta(i),
\end{aligned}
\end{split}
\end{equation*}
\noindent with $w(k) \hspace{-0.15cm}= \hspace{-0.15cm}\begin{bmatrix} x^{\T}(k) & \sum_{l=k-d_m}^{k-1} x^{\T}(l) & \sum_{l=k-d_M}^{k-d_m-1} x^{\T}(l) \end{bmatrix}^{\T}$ and $\eta(i)=x(i)-x(i-1)$. Matrices ${P} $ in $\mathbb{S}_{3n}^{+}$, ${W}_1$, ${W}_2$, ${Z}_1$, and ${Z}_2$ in $\mathbb{S}_{n}^{+}$ guarantee that the functional is positive definite. Consider the augmented vector
\begin{equation*}
    \begin{aligned}
\xi(k) &=
\left[\begin{matrix}
  x(k+1)^{\T} & x(k)^{\T} & x(k-d_m)^{\T}
\end{matrix}\right.\\
&\quad\quad
\left.\begin{matrix}
  {}x(k-d(k))^{\T} & x(k-d_M)^{\T} & v^{\T}
\end{matrix}\right]^{\T}
\end{aligned}
\end{equation*}
\noindent with $ v = \begin{bmatrix}
v_1^{\T} & v_2^{\T} & v_3^{\T}  
\end{bmatrix}^{\T}$ given by 
\begin{align*}
    v_1 &= \frac{1}{d_m+1}\sum_{l=k-d_m}^{k}x(l), \\
    v_2 &= \frac{1}{d(k)-d_m+1}\sum_{l=k-d(k)}^{k-d_m}x(l), \\
    v_3 &= \frac{1}{d_M-d(k)+1}\sum_{l=k-d_M}^{k-d(k)}x(l).
\end{align*}
\noindent From a procedure similar to the one in \cite{Seuret_2015}, the bound $\Delta V(k) \leq \xi^{\T}(k){\Phi}(d(k))\xi(k)$ (subject to the existence of matrix $X$ in $\reals^{2n\times2n}$ such that ${\Psi}_z \succ 0$) is obtained, where $\Phi(d(k))$ was first given in Theorem \ref{designtheorem}. Then, observe that satisfaction of  
\begin{equation}\label{eq:boundedLKF}
\begin{split}
     \xi^{\T}(k){\Phi}(d(k))\xi(k) \leq  w(u(k-d(k)),y(k-d(k))),
\end{split}
\end{equation}
\noindent subject to ${\Psi}_z \succ 0$ implies satisfaction of \eqref{eq:dissipativitymain}
with the functional $V(k)$ in \eqref{LKFFunc} and the supply rate \eqref{eq:suplydelay}. By using $y(k-d(k))=C x(k-d(k))$ and considering the extended vector $\zeta(k) = \begin{bmatrix}
    \xi^{\T}(k) & u^{\top}(k-d(k))
\end{bmatrix}^{\T}$, condition \eqref{eq:boundedLKF} can be equivalently rewritten in the form
\begin{equation}\label{eq:important}
    \zeta^{\T}(k)\left(\xbar{\Phi}(d(k))+\Upsilon\right)\zeta(k) \prec 0,
\end{equation}
where $\xbar{\Phi}(d(k))$ and $\Upsilon$ have been defined in Theorem \ref{designtheorem}. Next, we apply Lemma \ref{lemma:Finsler} with $m_{\zeta}=n$ and $n_{\zeta}=8n+m$. By noting that $\Gamma \zeta(k) = 0$, with $\Gamma = \begin{bmatrix} -I & A & 0_{n \times 6n} & B \end{bmatrix}$, relation \eqref{eq:important} is verified for all $\zeta \neq 0$ if  $\Gamma^{ \perp^{\T}}\left({\xbar{\Phi}}(d(k))+\Upsilon \right) \Gamma^{ \perp} \prec 0$, where the matrix $\Gamma^{\perp}$ given in Theorem \ref{designtheorem} is an orthogonal complement of $\Gamma$. Since ${\xbar{\Phi}}(d(k))$ is affine with respect to $d(k)$, the last inequality is negative definite if and only if it is negative definite for both $d(k)=d_m$ and $d(k)=d_M$, which leads to condition \eqref{main_matineq} in the Theorem. 

Furthermore, if at the same time we guarantee that the right-hand side of inequality \eqref{eq:boundedLKF} is negative definite with the proposed control law \eqref{eq:control}, then a sufficient condition to fulfill the asymptotic stability requirement $\Delta V(k)<0$ is obtained. By considering the extended vector $\lambda = \begin{bmatrix}
y^{\T}(k-d(k)) & u^{\T}(k-d(k)) 
\end{bmatrix}^{\T}$, such negativity test can be written as $\lambda^{\T} M_d \lambda < 0$, with 
\begin{equation*}
    M_d =  \begin{bmatrix}
Q & \star \\
S^{\T} & R 
\end{bmatrix}.
\end{equation*}
Consider the control law \eqref{eq:control} with gain $K=-R^{-1}S^{\T}$. By noting that $C_s  \lambda = 0$, with $C_s = \begin{bmatrix} S^{\T} & R \end{bmatrix}$, Lemma \ref{lemma:Finsler} can once again be applied. If there exists matrix $L_s \in \reals^{(p+m)\times m}$ such that $M_d+L_sC_s+C^{\T}_s L_s^{\T} \prec 0$, then $\lambda^{\T} M_d \lambda < 0$ is guaranteed for all $\lambda \neq 0$. Finally, sufficient condition \eqref{main_stab} is obtained by making $L_s= \begin{bmatrix}
\rho \ones_{m \times p} & -I_{m \times m}
\end{bmatrix}^{\T}$, with some auxiliary scalar $\rho$, thus completing the proof. 
\end{proof}

For better readability of the conditions, the developed design theorem was presented by using the LKF strategy from \cite{Seuret_2015}, which applies Wirtinger's summation inequalities and the reciprocally convex lemma for the LKF manipulation\footnote{The manipulation is omitted due to keeping the proof simple and focused on the main ideas. Details on the LKF manipulation can be found in \cite{Seuret_2015}.}. Nonetheless, in general, any other LKF (which also serves as storage function) and strategy for its manipulation (such as the more advanced Bessel-Legendre inequalities in \cite{Seuret_2018,LIU2017}) could have been chosen.

For a given gain $K$, stability analysis can also be carried out with the presented strategy by fixing $R=I$ and $S=-K^{\T}$ in the conditions. Furthermore, in this case, no special form to matrix $L_s$ in the proof of Theorem \ref{designtheorem} needs to be fixed to obtain linear conditions in the decision variables. The following Corollary can be stated. 

\begin{cor}\label{cor:stabanalysis}
For a given gain $K$, fixed $R=I$ and $S=-K^{\T}$, assume that there exist matrices ${P} $ in $\mathbb{S}_{3n}^{+}$, ${W}_1$, ${W}_2$, ${Z}_1$, ${Z}_2$, $N$ in $\mathbb{S}_{n}^{+}$, symmetric matrix $Q$ in $\reals^{p \times p}$, matrices ${X}$ in $\reals^{2n\times2n}$ and $L_s$ in $\reals^{(p+m)\times m}$ such that \eqref{main_matineq} and $M_d+L_s C_s+C^{\T}_s L_s^{\T} \prec 0$ are satisfied for $d=d_m$ and $d=d_M$. Then, the closed-loop system \eqref{clsystem} is asymptotically stable for any time-varying delay $d_m \leq d(k) \leq d_M$. 
\end{cor}

In the time-delay literature, very conservative operations are usually applied to transform stability analysis conditions into convex design ones. Notice that this is not a problem with the employed strategy since almost no conservatism is introduced when going from analysis (Corollary \ref{cor:stabanalysis}) to design (Theorem \ref{designtheorem}). In fact, the main condition \eqref{main_matineq} is the same in both cases, while the special form given to matrix $L_s$ in the proof of Theorem \ref{designtheorem} in order to obtain \eqref{main_stab} introduces very little conservatism. To see this, note that there always exist values of $\rho$, $Q$, and $S$ making $Q + \rho \ones_{p \times m} S^\T + S \rho \ones_{m \times p}$ negative definite. Then, by Schur complement, there always exist values of $R$ making \eqref{main_stab} negative definite. It should be noted, however, that such statements do not mean that there always exist values of these variables \textit{simultaneously} satisfying both conditions \eqref{main_matineq} and \eqref{main_stab} of the discussed theorem.


\section{Numerical example}\label{sec:simu}

Consider the network control system example from \cite{HU2007}, which was also recently studied in \cite{Lima2021}
\begin{equation*}
    \dot{x} = \begin{bmatrix}
        -0.80 & -0.01 \\
        \phantom{-}1.00 & \phantom{-}0.10
    \end{bmatrix} x + \begin{bmatrix}
        0.4 \\ 0.1
    \end{bmatrix} u.
\end{equation*}
By considering a sampling time of 0.5 seconds and a network induced delay $d(k)$, model \eqref{eq:Plant} is obtained with
\begin{equation*}
    A = \begin{bmatrix}
       0.6693 & -0.0042 \\
       0.04231 & \phantom{-}1.0501
    \end{bmatrix}, ~~~
    B = \begin{bmatrix}
        0.1647 \\ 0.0960
    \end{bmatrix}.
\end{equation*}
Let us initially consider the stability analysis problem. In \cite{HU2007}, a static state feedback (SSF) with $K = - \begin{bmatrix}
1.2625 & 1.2679
\end{bmatrix}$ is found, which guarantees stability for a maximum network delay $d_M=2$. By running Corollary \ref{cor:stabanalysis} with this same control, stability is ensured for $1 \leq d(k)\leq 3$. 

Now let us discuss the design problem. First, for the SOF design, consider an output matrix $C = \begin{bmatrix}
    0 & 1
\end{bmatrix}$. By running design Theorem \ref{designtheorem} with $\rho=-0.15$, we find SOF gain $K = -0.1498$ that guarantees stability for $1 \leq d(k) \leq 19$, a much larger bound than the one from \cite{HU2007}. Interestingly, by running the stability analysis Corollary \ref{cor:stabanalysis} with this same control we found that stability holds for the same bounds $1 \leq d(k) \leq 19$. This strongly suggest that, as highlighted earlier in the paper, in the proposed strategy there is small conservatism in the manipulations to achieve convex design conditions. Finally, for the SSF case, by running design Theorem \ref{designtheorem} again with $\rho=-0.16$, the stabilizing state feedback gain $K = \begin{bmatrix}
    -0.2260 &  -0.1656
\end{bmatrix}$ that guarantees asymptotic stability for $1 \leq d(k) \leq 21$ is found. Once again, analysis Corollary \ref{cor:stabanalysis} suggests small conservatism of the design conditions since for this same gain stability is guaranteed for a only slightly higher delay bound $d_M=22$. 

It is worth to comment that with the much more complicated strategy from \cite{Lima2021}, stability was ensured for a maximum delay of $d_M=7$, thus evidencing the potential of the proposed stabilizing strategy in this note.  

\section{Conclusion}\label{sec:conclu}
We presented new conditions for the SOF stabilization of linear discrete-time systems with input time-varying delays. The conditions were developed with the help of dissipativity definitions, along with the application of Lyapunov-Krasoviskii functionals as storage functions and Finsler's lemma. Ongoing work on continuous-time systems with both state and input delays is in the developments. Future work envisages extensions to the control of delayed linear parameter-varying (LPV) and nonlinear rational systems. Other Lyapunov-Krasoviskii functionals as storage functions along with more advanced summation inequalities for their manipulation can also be explored.  


\bibliographystyle{apacite}        
\bibliography{autosam}          


\end{document}